\UseRawInputEncoding 
\documentclass[12pt]{amsart}
\usepackage{}
\usepackage{amsmath}
\usepackage{amsfonts}
\usepackage{amssymb}
\usepackage[all,cmtip]{xy}           
\usepackage{bbding}
\usepackage{txfonts}
\usepackage[shortlabels]{enumitem}
\usepackage{ifpdf}
\ifpdf
\usepackage[colorlinks,final,backref=page,hyperindex]{hyperref}
\else
\usepackage[colorlinks,final,backref=page,hyperindex,hypertex]{hyperref}
\fi
\usepackage{tikz}
\usepackage[active]{srcltx}
\usepackage{graphicx}
\usepackage{bm}
\usepackage{url}
\usepackage{bbm}
\usepackage{diagbox}
\usepackage{makecell}
\usepackage{caption}
\usepackage{longtable}
\usepackage{makecell}
\usepackage{multirow}

\makeatletter

\newtheorem{theorem}{Theorem}[section]
\newtheorem{prop}[theorem]{Proposition}
\newtheorem{lemma}[theorem]{Lemma}
\newtheorem{coro}[theorem]{Corollary}
\newtheorem{prop-def}{Proposition-Definition}[section]
\newtheorem{conjecture}[theorem]{Conjecture}
\theoremstyle{definition}
\newtheorem{defn}[theorem]{Definition}

\newtheorem{remark}[theorem]{Remark}
\newtheorem{exam}[theorem]{Example}

\newcommand{\nc}{\newcommand}

\newcommand {\emptycomment}[1]{}

\nc{\delete}[1]{{}}
\nc{\mmargin}[1]{}

\nc{\mlabel}[1]{\label{#1}}  
\nc{\mcite}[1]{\cite{#1}}  
\nc{\mref}[1]{\ref{#1}}  
\nc{\meqref}[1]{\eqref{#1}}  
\nc{\mbibitem}[1]{\bibitem{#1}} 

\delete{
	\nc{\mlabel}[1]{\label{#1}  
		{\hfill \hspace{1cm}{\bf{{\ }\hfill(#1)}}}}
	\nc{\mcite}[1]{\cite{#1}{{\bf{{\ }(#1)}}}}  
	\nc{\mref}[1]{\ref{#1}{{\bf{{\ }(#1)}}}}  
	\nc{\meqref}[1]{\eqref{#1}{{\bf{{\ }(#1)}}}}  
	\nc{\mbibitem}[1]{\bibitem[\bf #1]{#1}} 
}

\nc{\vep}{\varepsilon}
\nc{\bin}[2]{ (_{\stackrel{\scs{#1}}{\scs{#2}}})}  
\nc{\binc}[2]{(\!\! \begin{array}{c} \scs{#1}\\
		\scs{#2} \end{array}\!\!)}  
\nc{\bincc}[2]{  ( {\scs{#1} \atop
		\vspace{-1cm}\scs{#2}} )}  
\nc{\oline}[1]{\overline{#1}}
\nc{\mapm}[1]{\lfloor\!|{#1}|\!\rfloor}
\nc{\bs}{\bar{S}}
\nc{\cast}{{\,\mbox{\raisebox{.8pt}{$\scriptstyle \circledast$}}\,}}
\nc{\la}{\longrightarrow}
\nc{\ot}{\otimes}
\nc{\rar}{\rightarrow}
\nc{\dar}{\downarrow}
\nc{\dap}[1]{\downarrow \rlap{$\scriptstyle{#1}$}}
\nc{\defeq}{\stackrel{\rm def}{=}}
\nc{\dis}[1]{\displaystyle{#1}}
\nc{\dotcup}{\ \displaystyle{\bigcup^\bullet}\ }
\nc{\hcm}{\ \hat{,}\ }
\nc{\hts}{\hat{\otimes}}
\nc{\hcirc}{\hat{\circ}}
\nc{\lleft}{[}
\nc{\lright}{]}
\nc{\curlyl}{\left \{ \begin{array}{c} {} \\ {} \end{array}
	\right .  \!\!\!\!\!\!\!}
\nc{\curlyr}{ \!\!\!\!\!\!\!
	\left . \begin{array}{c} {} \\ {} \end{array}
	\right \} }
\nc{\longmid}{\left | \begin{array}{c} {} \\ {} \end{array}
	\right . \!\!\!\!\!\!\!}
\nc{\ora}[1]{\stackrel{#1}{\rar}}
\nc{\ola}[1]{\stackrel{#1}{\la}}
\nc{\scs}[1]{\scriptstyle{#1}} \nc{\mrm}[1]{{\rm #1}}
\nc{\dirlim}{\displaystyle{\lim_{\longrightarrow}}\,}
\nc{\invlim}{\displaystyle{\lim_{\longleftarrow}}\,}
\nc{\dislim}[1]{\displaystyle{\lim_{#1}}} \nc{\colim}{\mrm{colim}}
\nc{\mvp}{\vspace{0.3cm}} \nc{\tk}{^{(k)}} \nc{\tp}{^\prime}
\nc{\ttp}{^{\prime\prime}} \nc{\svp}{\vspace{2cm}}
\nc{\vp}{\vspace{8cm}}
\nc{\modg}[1]{\!<\!\!{#1}\!\!>}
\nc{\intg}[1]{F_C(#1)}
\nc{\lmodg}{\!<\!\!}
\nc{\rmodg}{\!\!>\!}
\nc{\cpi}{\widehat{\Pi}}
\nc{\ssha}{{\mbox{\cyrs X}}} 
\nc{\tsha}{{\mbox{\cyrt X}}}
\nc{\shpr}{\diamond}    
\nc{\labs}{\mid\!}
\nc{\rabs}{\!\mid}
\nc{\btr}{\blacktriangleright}

\nc{\ad}{\mrm{ad}}
\nc{\rRB}{\mathsf{rRB}}
\nc{\cocrRB}{\mathsf{cocrRB}}
\nc{\PH}{\mathsf{PH}}
\nc{\cocPH}{\mathsf{cocPH}}
\nc{\ann}{\mrm{ann}}
\nc{\Aut}{\mrm{Aut}}
\nc{\Der}{\mrm{Der}}
\nc{\Sym}{\mrm{Sym}}
\nc{\br}{\mrm{bre}}
\nc{\can}{\mrm{can}}
\nc{\Cont}{\mrm{Cont}}
\nc{\rchar}{\mrm{char}}
\nc{\cok}{\mrm{coker}}
\nc{\de}{\mrm{dep}}
\nc{\dtf}{{R-{\rm tf}}}
\nc{\dtor}{{R-{\rm tor}}}

\nc{\Dif}{\mrm{Diff}}
\nc{\Div}{\mrm{Div}}
\nc{\End}{\mrm{End}}
\nc{\Ext}{\mrm{Ext}}
\nc{\Fil}{\mrm{Fil}}
\nc{\Fr}{\mrm{Fr}}
\nc{\Frob}{\mrm{Frob}}
\nc{\Gal}{\mrm{Gal}}
\nc{\GL}{\mrm{GL}}
\nc{\Gr}{\mrm{Gr}}
\nc{\Hom}{\mrm{Hom}}
\nc{\Hoch}{\mrm{Hoch}}
\nc{\hsr}{\mrm{H}}
\nc{\hpol}{\mrm{HP}}
\nc{\id}{\mrm{id}}
\nc{\im}{\mrm{im}}
\nc{\inv}{\mrm{inv}}
\nc{\Id}{\mrm{Id}}
\nc{\ID}{\mrm{ID}}
\nc{\Irr}{\mrm{Irr}}
\nc{\incl}{\mrm{incl}}
\nc{\length}{\mrm{length}}
\nc{\NLSW}{\mrm{NLSW}}
\nc{\Lie}{\mrm{Lie}}
\nc{\mchar}{\rm char}
\nc{\mpart}{\mrm{part}}
\nc{\ql}{{\QQ_\ell}}
\nc{\qp}{{\QQ_p}}
\nc{\rank}{\mrm{rank}}
\nc{\rcot}{\mrm{cot}}
\nc{\rdef}{\mrm{def}}
\nc{\rdiv}{{\rm div}}
\nc{\rtf}{{\rm tf}}
\nc{\rtor}{{\rm tor}}
\nc{\res}{\mrm{res}}
\nc{\SL}{\mrm{SL}}
\nc{\Spec}{\mrm{Spec}}
\nc{\tor}{\mrm{tor}}
\nc{\Tr}{\mrm{Tr}}
\nc{\tr}{\mrm{tr}}
\nc{\wt}{\mrm{wt}}

\nc{\bfk}{{\bf k}}
\nc{\bfone}{{\bf 1}}
\nc{\bfzero}{{\bf 0}}
\nc{\detail}{\marginpar{\bf More detail}
	\noindent{\bf Need more detail!}
	\svp}
\nc{\gap}{\marginpar{\bf Incomplete}\noindent{\bf Incomplete!!}
	\svp}
\nc{\FMod}{\mathbf{FMod}}
\nc{\Int}{\mathbf{Int}}
\nc{\Mon}{\mathbf{Mon}}
\nc{\remarks}{\noindent{\bf Remarks: }}
\nc{\Rep}{\mathbf{Rep}}
\nc{\Rings}{\mathbf{Rings}}
\nc{\Sets}{\mathbf{Sets}}
\nc{\Diff}{\mathbf{Diff}}
\nc{\Inte}{\mathbf{Inte}}
\nc{\U}{\mathbf{U}}

\nc{\BA}{{\mathbb A}}   \nc{\CC}{{\mathbb C}}
\nc{\DD}{{\mathbb D}}   \nc{\EE}{{\mathbb E}}
\nc{\FF}{{\mathbb F}}   \nc{\GG}{{\mathbb G}}
\nc{\HH}{{\mathbb H}}   \nc{\LL}{{\mathbb L}}
\nc{\NN}{{\mathbb N}}   \nc{\PP}{{\mathbb P}}
\nc{\QQ}{{\mathbb Q}}   \nc{\RR}{{\mathbb R}}
\nc{\TT}{{\mathbb T}}   \nc{\VV}{{\mathbb V}}
\nc{\ZZ}{{\mathbb Z}}   \nc{\TP}{\widetilde{P}}

\nc{\cala}{{\mathcal A}}    \nc{\calc}{{\mathcal C}}
\nc{\cald}{\mathcal{D}}     \nc{\cale}{{\mathcal E}}
\nc{\calf}{{\mathcal F}}    \nc{\calg}{{\mathcal G}}
\nc{\calh}{{\mathcal H}}    \nc{\cali}{{\mathcal I}}
\nc{\call}{{\mathcal L}}    \nc{\calm}{{\mathcal M}}
\nc{\caln}{{\mathcal N}}    \nc{\calo}{{\mathcal O}}
\nc{\calp}{{\mathcal P}}    \nc{\calr}{{\mathcal R}}

\nc{\cals}{{\mathcal S}}    \nc{\calt}{{\Omega}}
\nc{\calv}{{\mathcal V}}    \nc{\calu}{{\mathcal U}}
\nc{\calw}{{\mathcal W}}
\nc{\calx}{{\mathcal X}}

\nc{\fraka}{{\mathfrak a}}
\nc{\frakb}{\mathfrak{b}}
\nc{\frakg}{{\frak g}}
\nc{\frakh}{{\frak h}}
\nc{\frakl}{{\frak l}}
\nc{\fraks}{{\frak s}}
\nc{\frakB}{{\frak B}}
\nc{\frakm}{{\frak m}}
\nc{\frakM}{{\frak M}}
\nc{\frakp}{{\frak p}}
\nc{\frakW}{{\frak W}}
\nc{\frakX}{{\frak X}}
\nc{\frakS}{{\frak S}}
\nc{\frakA}{{\frak A}}
\nc{\frakx}{{\frakx}}

\nc{\ynr}[1]{\textcolor{orange}{\underline{Yunnan:}#1 }}

\nc{\lir}[1]{\textcolor{red}{\underline{Li:}#1 }}

\delete{
	\input cyracc.def

	\allowdisplaybreaks
	
	\numberwithin{equation}{section}
	
}

\begin{document}

\title[Towards the classification of finite-dimensional DGCAs]{Towards the classification of finite-dimensional diagonally graded commutative 
algebras}

\author{Yunnan Li}
\address{School of Mathematics and Information Science, Guangzhou University, Guangzhou 510006, China}
\email{ynli@gzhu.edu.cn}

\author{Shi Yu}
\address{School of Mathematics and Information Science, Guangzhou University, Guangzhou 510006, China}
\email{2112115054@e.gzhu.edu.cn}



\begin{abstract}
Any finite-dimensional commutative (associative) graded algebra with all nonzero homogeneous subspaces one-dimensional is defined by a symmetric coefficient matrix. This algebraic structure gives a basic kind of $A$-graded algebras originally studied by Arnold. In this paper, we call them diagonally graded commutative algebras (DGCAs) and verify that the isomorphism classes of DGCAs of dimension $\leq 7$ over an arbitrary field are in bijection with the equivalence classes consisting of coefficient matrices with the same distribution of nonzero entries, while dramatically there may be infinitely many isomorphism classes of dimension $n$ corresponding to one equivalence class of coefficient matrices when $n\geq 8$.

Furthermore, we adopt the Skjelbred-Sund method of central extensions to study the isomorphism classes of DGCAs, and associate any DGCA with a undirected simple graph to explicitly describe its corresponding second (graded) commutative cohomology group as an affine variety.
\end{abstract}

\keywords{commutative algebra, diagonally graded algebra, isomorphism class, central extension\\
\qquad 2020 Mathematics Subject Classification. 13A02, 13E10, 14L30}

\maketitle

\tableofcontents

\allowdisplaybreaks

\section{Introduction}
\smallskip
Associative algebra, as a fundamental algebraic structure, has been studied extensively for many years, with the algebraic classification of associative algebras as an ancient and longstanding topic worthy of discussion. However, since this algebraic system is huge and extraordinarily complicated, the classification of associative algebras up to isomorphism is far from being completed.

At present, there have been many achievements about the classification of associative algebras of finite dimension.
Especially, nilpotent associative algebras form a large class of associative algebras, which has fruitful classification results. For example, de Graaf completed the classification of 4-dimensional nilpotent associative algebras over arbitrary fields in \cite{De} by constructing central extensions of smaller dimensional nilpotent associative algebras. Recently, Kaygorodov et al. discussed the classification of 5-dimensional nilpotent commutative associative algebras over the complex field also based on the Skjelbred-Sund method of central extensions \cite{KRS}. In addition, Kaygorodov et al. specifically gave a geometric classification of $n$-dimensional nilpotent, commutative nilpotent and anticommutative nilpotent algebras in \cite{KKl}.

On the other hand, the study of graded algebras is an active area, applied in various fields of mathematics and physics. The gradings on classes of algebras provide a powerful tool for understanding and analyzing their algebraic structures. For example, the concept of grading was widely used in the study of Lie algebras (e.g.~\cite{Kac,Ma,GJ}), and also other related algebraic structures, such as associative algebras~\cite{BSZ}, pre-Lie algebras~\cite{Ch,KCB} and Jordan algebras~\cite{BSz}.
In \cite{Ar} Arnold introduced the notion of {\it $A$-graded algebra}, and mainly studied the classification of commutative $A$-graded algebras with 3 multiplicative generators and all nontrivial homogeneous subspaces of dimension 1. Also, Arnold initiated the study of $A$-graded ideals in \cite{Ar}, and found that under certain circumstances, this ideal structure can be encoded as a continued fraction. Subsequently, the authors in \cite{KOR} and \cite{KPR} supplemented classification method of $A$-graded algebras with 3 generators in detail. General study of $A$-graded algebras and $A$-graded ideals can be found in e.g.~\cite{St0,St,PS}.

In this paper, we study a class of finite-dimensional commutative associative $\mathbb N^+$-graded algebras (consequently nilpotent), with all nonzero homogeneous subspaces being $1$-dimensional, and the corresponding graded multiplication can be characterized by a symmetric coefficient matrix.
This algebraic structure is actually one basic kind of $A$-graded algebras with the first generator of degree 1, and we specially call them {\it diagonally graded commutative algebras} (DGCAs). Based on such a characterization by coefficient matrices, we try to attack the classification problem of finite-dimensional DGCAs.

It is noted that if commutative associative algebras are endowed with graded structure, then the algebras originally classified in one isomorphism class are often further divided into several graded isomorphism classes. In addition, the classification of one algebraic structure is closely related to which field it is defined over. For example, the classification of real simple Lie algebras can be obtained by classifying involutions of compact Lie algebras based on the classification of complex simple Lie algebras. By contrast, the classification of DGCAs of low dimension has a prominent feature, that is, the isomorphism classes over any prime field are already indivisible, and do not subdivide over arbitrary field extensions.

Based on our initial classification results in low dimension,
we expect that the classification of DGCAs is simply determined by the distribution of nonzero entries in their coefficient matrices, namely the existence of a one-to-one correspondence between the isomorphism classes of finite-dimensional DGCAs and certain equivalence classes of coefficient matrices.
However, through the study of central extensions of DGCAs, we find that such an expectation collapses once the dimension $n>7$.
As a matter of fact, there appear infinitely many isomorphism classes with the same distribution of nonzero entries in their coefficient matrices.
One can compare such a variation with the classification result of Poonen
in \cite{Po} saying that there are only finitely many isomorphism classes of commutative associative unital algebras of dimension up to 6 over an algebraically closed field, but infinitely many of dimension $>6$.

Concretely, we interpret any $n$-dimensional DGCA $\frakp$ as a $1$-dimensional central extension of a DGCA $\frakp'$ of dimension $n-1$ (Proposition~\ref{prop:DGCA_ext}), and describe the corresponding second (graded) commutative cohomology group of $\frakp'$ explicitly as an affine variety. Any of its orbit closures is indeed a toric variety, under the action of the automorphism group $\Aut(\frakp')$ of $\frakp'$.

This paper is organized by the upcoming three sections. In Section~\ref{sec:preliminaries}, we introduce the notion of diagonally graded commutative algebras and their isomorphism classes, then state our classification result on DGCAs (Theorem~\ref{th:classification}).
In Section~\ref{sec:central_extension}, we adopt the Skjelbred-Sund method to study  central extensions of DGCAs, and show that the nontrivial isomorphism classes of central extensions of a DGCA $\frakp'$ are in bijection with the $\Aut(\frakp')$-orbits in the Grassmannian of $1$-dimensional subspaces of $2$-cocycles (Theorem~\ref{th:isom}). After that, we construct a kind of graphs associated with coefficient matrices to explicitly characterize the second (graded) commutative cohomology groups (Theorem~\ref{th:cocyle_graph}), which control the isomorphism classes of central extensions of DGCAs.
In particular, we obtain a criterion when all nontrivial central extensions of a DGCA form exactly one isomorphism class (Theorem~\ref{th:orbit}). In Section~\ref{sec:low_dim},
we completely list all isomorphism classes of DGCAs of dimension up to 5.
Also, some isomorphism classes of DGCAs of dimension 6 and 7 are presented in the same manner as an illustration.

\medskip\noindent
{\bf Notations.}
Let $\mathbb N$ (resp. $\mathbb N^+$) be the set of nonnegative (resp. positive) integers.
All the matrices, vector spaces, algebras and linear maps are taken over an algebraically closed field $\mathbbm k$ unless otherwise specified. Also, we denote $\mathbbm k^*=\mathbbm k\backslash\{0\}$ and $[n]=\{1,\dots,n\}$ for any positive integer $n$.

\section{Diagonally graded commutative algebras and their isomorphism classes}\label{sec:preliminaries}
\begin{defn}
An {\bf associative algebra} $A$ is a vector space with a bilinear multiplication $(x,y)\mapsto xy$ such that
\begin{eqnarray}\label{eq:ass_alg}
x(yz)=(xy)z,\quad\forall x,y,z\in A.
\end{eqnarray}
When $xy=yx,\forall x,y\in A$, we say that the associative algebra $A$ is a {\bf commutative (associative) algebra}.

If the (commutative) algebra $A$ is an $\mathbb N$-graded vector space $A=\oplus_{n\geq0} A_n$ such that $A_nA_m\subset A_{n+m}$ for any $n,m\geq0$, then we call $A$ a {\bf (commutative) graded algebra}.

An {\bf isomorphism} between two (graded) algebras $A$ and $A'$ is a (graded) linear isomorphism $\varphi: A\to A'$ such that
\begin{eqnarray}\label{eq:alg_iso}
\varphi(xy)=\varphi(x)\varphi(y),\quad\forall x,y\in A.
\end{eqnarray}
\end{defn}
Clearly, any finite-dimensional $\mathbb N^+$-graded associative algebra $A=\oplus_{n\geq1} A_n$ is {\it nilpotent} by its graded multiplication.

From now on, we focus on an $n$-dimensional $\mathbb N^+$-graded vector space $\mathfrak p$ with a fixed homogeneous linear basis $\{p_{i}\}_{i\in [n]}$ such that $\deg p_i=i$, so call $\mathfrak p$ a {\bf diagonally graded} vector space with all its nonzero homogeneous subspaces one-dimensional. If $\mathfrak p$ has a graded bilinear multiplication $(x,y)\mapsto xy,\forall x,y\in{\mathfrak p}$, then there exists an $n\times n$ coefficient matrix $C=(c_{ij})_{i,j\in [n]}$ such that
\begin{eqnarray}\label{eq:GAA}
p_{i}p_{j}=c_{ij}p_{i+j},
\end{eqnarray}
where we take $c_{ij}=0$ and set $p_{i+j}=0$ if $i+j>n$ by convention. From the relation \eqref{eq:ass_alg}, one can see that $\mathfrak p$ is a graded associative algebra, if and only if
\begin{eqnarray}\label{eq:ass_coef}
c_{jk}c_{i,\,j+k}=c_{ij}c_{i+j,\,k},\quad\forall\,i,j,k\geq1.
\end{eqnarray}
In particular, $\frakp$ is commutative if and only if the coefficient matrix $C$ is a symmetric matrix.

In order to classify the graded associative algebraic structures on $\mathfrak p$, we need to find all their isomorphism classes.

\begin{theorem}\label{th:GAA_iso}
Two $n$-dimensional graded associative algebras underlying $\frakp$ with coefficient matrices $C$ and $C'$ respectively
are isomorphic under the graded algebra isomorphism $\varphi$, if and only if there exists a tuple $(b_{i})_{i\in [n]}$ of nonzero coefficients such that
\begin{eqnarray}\label{eq:GAA_iso}
b_{i+j}c_{ij}=b_{i}b_{j}c'_{ij},\quad\forall i,j\geq1,
\end{eqnarray}
where $C=(c_{ij})_{i,j\in [n]}$ and $C'=(c'_{ij})_{i,j\in [n]}$. Here set $b_{i+j}=0$, when $i+j>n$ by convention.
\end{theorem}

\begin{proof}
As $\varphi$ is a graded linear automorphism of $\mathfrak p$, there are nonzero coefficients $b_i\in\mathbbm k,\ i\in [n]$, such that $\varphi(p_{i})=b_{i}p_{i}$, where $\{p_{i}\}_{i\in [n]}$ is the fixed homogeneous linear basis of $\mathfrak p$. Substituting it and Eq.~\eqref{eq:GAA} into Eq.~\eqref{eq:alg_iso}, one can easily check that the relation \eqref{eq:GAA_iso} holds.
\end{proof}

\begin{defn}\label{defn:dgca}
The above commutative graded algebra $\frakp$ associated with a coefficient matrix $C$ is called a {\bf diagonally graded commutative algebra} (with the abbreviation ``DGCA'') and denoted by $\frakp(C)$.

For any $n$-dimensional DGCA  $\frakp(C)$, there is an associated $(n-1)$-dimensional
DGCA $\bar\frakp(\bar C)$ as its {\bf truncation}, where $\bar\frakp$ is the graded subspace of $\frakp$ spanned by $\{p_{i}\}_{i\in [n-1]}$ and $\bar C=(\bar c_{ij})_{i,j\in [n-1]}$ is the $(n-1)\times(n-1)$ symmetric matrix defined by
$$\bar c_{ij}=\begin{cases}
c_{ij},&i+j<n,\\
0,&{\rm otherwise}.
\end{cases}$$
Correspondingly, we call $\bar C$ the {\bf truncation} of $C$.
\end{defn}

The notion of diagonally graded algebra was originally introduced by Arnold in \cite{Ar}, where he called them $A$-graded algebras and actually classified all unital DGCAs with three multiplicative generators (and $p_0=1$). Using our notation of DGCAs, we can obtain the following result concerning about multiplicative generators of a DGCA.
\begin{prop}
For any $n$-dimensional DGCA  $\frakp(C)$, let
$$J_C=\{k\in[n]\,|\,c_{ij}=0\ \mbox{if}\ i+j=k\}.$$
Then $\{p_k\}_{k\in J_C}$ is a minimal set of multiplicative generators of $\frakp(C)$.
\end{prop}
\begin{proof}
Indeed, we must have $1\in J_C$, so $J_C$ is nonempty. According to Eq.~\eqref{eq:GAA}, for any $k\in [n]$, $p_k$ can be expressed as a polynomial of $p_i$'s with smaller $i$, if and only if $k\notin J_C$. Therefore, $\{p_k\}_{k\in J_C}$ is a minimal set of multiplicative generators of $\frakp(C)$.
\end{proof}

\begin{defn}
Let $\mathcal M_n$ be the set of $n\times n$ symmetric matrices $C=(c_{ij})_{i,j\in [n]}$ over $\mathbbm k$ satisfying Eq.~\eqref{eq:ass_coef} and $c_{ij}=0$ if $i+j>n$.

Introduce the equivalence relation $\sim$ on $\mathcal M_n$ such that $C\sim C'$ if and only if the distributions of nonzero entries in $C$ and $C'$ coincide.

Especially, each equivalence class in $\mathcal M_n$ with respect to $\sim$ has a unique $(0,1)$-matrix as its canonical representative.
\end{defn}

At first we provide the following concise classification of DGCAs of low dimension (see Section~\ref{sec:low_dim}).
\begin{theorem}\label{th:classification}
When $n\leq 7$, two $n$-dimensional DGCAs $\frakp(C)$ and $\frakp(C')$ are isomorphic if and only if $C\sim C'$. Namely, the set of isomorphism classes of $n$-dimensional DGCAs is in bijection with the quotient set $\mathcal M_n/\sim$ of $\mathcal M_n$ modulo $\sim$.
In particular, all $(0,1)$-matrices in $\mathcal M_n$  as the coefficient matrices define a canonical representative system of isomorphism classes of DGCAs.
\end{theorem}
\begin{proof}
The stated classification of DGCAs of low dimension can be checked directly, and we will do it in Section~\ref{sec:low_dim}.
\end{proof}

Unfortunately, the complexity of classification of DGCAs in general is far beyond our expectation.
If $\frakp(C)$ and $\frakp(C')$ are isomorphic to each other, then the positions of nonzero entries of $C$ and $C'$ are the same by Eq.~\eqref{eq:GAA_iso}, so $C\sim C'$. But the opposite is not necessarily true when $n>7$. After some trial, we find a counterexample of dimension $n=8$ as follows.
\begin{exam}\label{ex:counter}
Let $C=(c_{ij})_{i,j\in [8]}$ be any $8\times 8$ symmetric matrix with nonzero entries
$$c_{24}=c_{42},\ c_{33},\ c_{25}=c_{52},\ c_{34}=c_{43},\ c_{35}=c_{53},\ c_{44}.$$
It is clear that $C\in\mathcal M_8$. Let $C'$ be the $(0,1)$-matrix such that $C\sim C'$.

According to Eq.~\eqref{eq:GAA_iso}, if $\frakp(C)\cong\frakp(C')$, then there exists $(b_{i})_{i\in [8]}\in (\mathbbm k^*)^{\times 8}$ such that
\begin{align*}
b_6&=b_2b_4c_{24}=b_3^2c_{33},\\
b_7&=b_2b_5c_{25}=b_3b_4c_{34},\\
b_8&=b_3b_5c_{35}=b_4^2c_{44}.
\end{align*}
Hence, we have $b_4=\dfrac{b_3^2c_{33}}{b_2c_{24}}$ and then
$$b_5=\dfrac{b_3b_4c_{34}}{b_2c_{25}}=\dfrac{b_3c_{34}}{b_2c_{25}}\cdot\dfrac{b_3^2c_{33}}{b_2c_{24}}
=\dfrac{b_3^3c_{33}c_{34}}{b_2^2c_{24}c_{25}}.$$
Also,
$$ b_5=\dfrac{b_4^2c_{44}}{b_3c_{35}}=\dfrac{c_{44}}{b_3c_{35}}\cdot\left(\dfrac{b_3^2c_{33}}{b_2c_{24}}\right)^2
=\dfrac{b_3^3c_{33}^2c_{44}}{b_2^2c_{24}^2c_{35}}.$$
Therefore, we obtain the constraint $c_{25}c_{33}c_{44}=c_{24}c_{34}c_{35}$, which means that the only condition $C\sim C'$ can not guarantee the isomorphism $\frakp(C)\cong\frakp(C')$.

In fact, there have already been infinitely many isomorphism classes of DGCAs of dimension 8 (see Example~\ref{ex:counter'}).
As a result, we need to search for more elaborate technique to study the isomorphism classes of DGCAs.
\end{exam}

\section{Classification by the Skjelbred-Sund method of central extensions}\label{sec:central_extension}
The Skjelbred-Sund method is well-known for classifying nilpotent Lie algebras \cite{SS,De1}.
In this section, we study the classification problem of central extensions of finite-dimensional DGCAs by a graded analogue of such a method, which has also been adopted to the classification of nilpotent associative algebras of small dimension \cite{De}; see also the recently published survey \cite{Kay} for other cases where it was developed.

\subsection{Central extensions of diagonally graded commutative algebras}
First we recall some terminology.
Let $A$ be a commutative associative algebra, $V$ a vector space, and $\theta: A\otimes A\to V$ a symmetric bilinear
map.

Set $A_\theta = A\oplus V$ . For $a, b \in A, v,w \in V$, we define
\begin{equation}\label{eq:central_extension}
(a+v)\cdot_\theta(b+w)= ab+\theta(a,b).
\end{equation}
Then $(A_\theta,\cdot_\theta)$ is a commutative associative algebra if and only if
$\theta(ab, c) = \theta(a, bc)$
for all $a, b, c \in A$. All these symmetric invariant bilinear maps $\theta$ are exactly the $2$-cocycles of $A$ with coefficients in the trivial $A$-module $V$, when one studies the commutative cohomology of $A$ \cite{Ha}. The space of all $2$-cocycles is denoted by
$Z^2(A,V)$. The algebra $A_\theta$ is called a {\bf central extension} of $A$
by $V$ (such that $A_\theta V = V A_\theta = 0$).

Let $f:A \to V$ be a linear map, and define $\delta f(a,b) = f(ab)$. Then $\delta f$ is a
2-coboundary. The space of all 2-coboundaries is denoted by $B^2(A,V)$. Note that $A_\theta\cong
A_{\theta+\delta f}$ and particularly $A_{\delta f}\cong A\times V$, the direct product of algebras $A$ and $V$, which we call the {\bf trivial} central extension. Therefore, we consider the second cohomology space $H^2(A,V)=Z^2(A,V)/B^2(A,V)$.

Correspondingly, we can define the graded version of central extension $A_\theta$ of a commutative associative graded algebra $A$ by a graded module $V$, requiring that the $2$-cocycle $\theta$ is graded.
Also, we use the distinct notations $Z_g^2(A,V),B_g^2(A,V)$ and $H_g^2(A,V)$ for the graded version of the commutative cohomology.
Now for our main target, namely DGCAs, we actually have
\begin{prop}\label{prop:DGCA_ext}
Given any $n$-dimensional DGCA $\frakp(C)$ with the fixed homogeneous linear basis $\{p_{i}\}_{i\in [n]}$, let $\bar\frakp(\bar C)$ be its $(n-1)$-dimensional truncation as in Definition~\ref{defn:dgca}. Then
$\frakp(C)$ is a central extension of $\bar\frakp(\bar C)$ by the $1$-dimensional homogeneous vector space $V$ spanned by $p_n$.
\end{prop}
\begin{proof}
In fact, we can define the graded $2$-cocycle $\theta:\bar\frakp(\bar C)\otimes\bar\frakp(\bar C)\to V$ by
\begin{equation*}
\theta(p_i,p_j)=\delta_{i+j,\,n}c_{ij}p_n,\quad \forall i,j\in [n-1],
\end{equation*}
whose cocycle condition is guaranteed by Eq.~\eqref{eq:ass_coef}.
Then $\frakp(C)$ is naturally identified with the central extension $\bar\frakp(\bar C)_\theta$ as commutative associative graded algebras.
\end{proof}

Conversely, for any $(n-1)$-dimensional DGCA $\frakp(C)$ and $1$-dimensional homogeneous vector space $V$ of degree $n$, the central extension $\frakp(C)_\theta$ with $\theta\in Z_g^2(\frakp(C),V)$ is clearly an $n$-dimensional DGCA.
In such a specified situation, any graded linear map from $\frakp(C)$ to $V$ is zero, so $B_g^2(\frakp(C),V)=0$ and $H_g^2(\frakp(C),V)=Z_g^2(\frakp(C),V)$.
Consequently, the trivial isomorphism class of central extensions of $\frakp(C)$ only contains $\frakp(C)_0$.
Also, for any $\theta\in Z_g^2(\frakp(C),V)$
we have coefficients $\theta_i\in\mathbbm k,\ i\in [n-1]$, such that
\begin{equation}\label{eq:cocycle}
\theta(p_i,p_j)=\delta_{i+j,\,n}\theta_ip_n,\quad\forall i,j\in [n-1].
\end{equation}

Let $\sigma:\frakp(C)\to\frakp(C')$ be a graded algebra isomorphism of two $(n-1)$-dimensional DGCAs, and $V=\mathbbm k p_n$ be the $1$-dimensional homogeneous vector space of degree $n$. Define
$\theta_\sigma\in Z_g^2(\frakp(C),V)$ by
\begin{equation}\label{eq:cocycle_action}
\theta_\sigma(x,y)=\theta(\sigma(x),\sigma(y))
\end{equation}
for any $\theta\in Z_g^2(\frakp(C'),V)$ and $x,y\in \frakp(C)$.

Let $\Aut(\frakp(C))$ be the graded automorphism group of $\frakp(C)$, consisting of diagonal automorphisms $\varphi$ such that
$\varphi(p_{i})=b_{i}p_{i}$ for certain $b_i\in\mathbbm k^*\ (i\in [n])$
by Theorem~\ref{th:GAA_iso}, so we write $\varphi=(b_{i})_{i\in [n]}$ for convenience. Then Eq.~\eqref{eq:cocycle_action} defines a right action of $\Aut(\frakp(C))$ on $Z_g^2(\frakp(C),V)$. 

As a special case of Theorem~\ref{th:GAA_iso}, we have the following result.
\begin{prop}\label{prop:aut}
Given any $C=(c_{ij})_{i,j\in [n]}\in \mathcal M_n$,
\begin{equation}\label{eq:aut}
\Aut(\frakp(C))=\left\{(b_i)_{i\in [n]}\in(\mathbbm k^*)^{\times n}\,\big|\, b_{i+j}=b_ib_j\ {\rm if}\ c_{ij}\neq 0\right\}.
\end{equation}
In particular, if $C\sim C'$ in $\mathcal M_n$, then $\Aut(\frakp(C))=\Aut(\frakp(C'))$.

Morevoer, $\Aut(\frakp(C))$ is a closed subgroup of the torus $(\mathbbm k^*)^{\times n}$
of dimension $n- n(C)+t(C)$, where
$$n(C)=\#\{(i,j)\in(\mathbb N^+)^{\times 2}\,|\,c_{ij}\neq 0,\,i\leq j\}$$
and $t(C)$ is the number of non-redundant equalities in Eq.~\eqref{eq:ass_coef} satisfying
$c_{ij}c_{i+j,\,k}\neq0$.
\end{prop}
\begin{proof}
First by Eq.~\eqref{eq:GAA_iso}, we obtain the stated description of $\Aut(\frakp(C))$ to  see that it is a closed subgroup of the torus $(\mathbbm k^*)^{\times n}$ subjecting to certain polynomial constraints.

It remains to show that there are exactly $n(C)-t(C)$ non-redundant polynomial equalities therein by the symmetry of $C$, so the dimension of $\Aut(\frakp(C))$ is obtained. In fact, for any $i,j,k\geq 1$ such that
$$c_{ij}c_{i+j,\,k}\stackrel{\eqref{eq:ass_coef}}{=}c_{jk}c_{i,\,j+k}\neq0,$$
we have the following 4 polynomial equalities
$$b_{i+j}=b_ib_j,\,b_{j+k}=b_jb_k,\,b_{i+j+k}=b_{i+j}b_k,\,b_{i+j+k}=b_ib_{j+k},$$
any 3 of which imply the remaining one. On the other hand,
any equality in Eq.~\eqref{eq:ass_coef} involving zero coefficients will not produce redundance.
\end{proof}

\begin{exam}
Given $b\in\mathbbm k^*$, by Eq.~\eqref{eq:aut} we have $\varphi_b=(b^i)_{i\in [n]}\in \Aut(\frakp(C))$ for any $C\in\mathcal M_n$.
\end{exam}

\begin{exam}
Let $C=(c_{ij})_{i,j\in [6]}$ be the symmetric matrix in $\mathcal M_6$ with
$$c_{11}=c_{12}=c_{21}=c_{13}=c_{31}=c_{22}=c_{14}=c_{41}=c_{23}=c_{32}=1$$
and zeros elsewhere. Then Eq.~\eqref{eq:aut} says that
$$\Aut(\frakp(C))=\left\{(b_i)_{i\in [6]}\in(\mathbbm k^*)^{\times 6}\,\big|\, b_2=b_1^2,\,b_3=b_1b_2,\,b_4=b_1b_3=b_2^2,\,b_5=b_1b_4=b_2b_3\right\}.$$
That is, $\Aut(\frakp(C))=\left\{(b_i)_{i\in [6]}\in(\mathbbm k^*)^{\times 6}\,\big|\, b_i=b_1^i,\,\forall i\in [5]\right\}$ is of dimension 2. On the other hand,
$n(C)=6$, and $t(C)=2$ counts the following non-redundant equalities
$$c_{12}c_{13}=c_{11}c_{22},\quad c_{11}c_{23}=c_{13}c_{14},$$
together to deduce the third one, $c_{12}c_{23}=c_{14}c_{22}$. So we also obtain that $$\dim\Aut(\frakp(C))=6-n(C)+t(C)=2.$$
\end{exam}

\begin{theorem}\label{th:isom}
Two central extensions $\frakp(C)_\gamma$ and $\frakp(C')_\theta$ for $\gamma\in Z_g^2(\frakp(C),V)$ and $\theta\in Z_g^2(\frakp(C'),V)$ are isomorphic if and only if there exists a graded algebra isomorphism $\sigma:\frakp(C)\to\frakp(C')$ such that $\theta_\sigma$ and $\gamma$ are colinear in $Z_g^2(\frakp(C),V)$.

In particular, $\frakp(C)_\theta\cong \frakp(C)_\gamma$ for nonzero $2$-cocycles $\theta,\gamma\in Z_g^2(\frakp(C),V)$ if and only if $[\theta],[\gamma]$ lie in the same $\Aut(\frakp(C))$-orbit of
the Grassmannian of $1$-dimensional subspaces of $Z_g^2(\frakp(C),V)$, where $[\theta]={\rm span}_\mathbbm k\{\theta\}$.
\end{theorem}
\begin{proof}
As discussed in Prop.~\ref{prop:DGCA_ext}, we can rewrite two matrices $C_\gamma$ and $C'_\theta$ in $\mathcal M_n$ augmented from $C$ and $C'$ respectively as
$D=(d_{ij})_{i,j\in [n]}$ and $D'=(d'_{ij})_{i,j\in [n]}$ instead, then
$$d_{ij}=\begin{cases}
c_{ij},& i+j<n,\\
\gamma_i, & i+j=n,\\
0,& {\rm otherwise},
\end{cases}\quad d'_{ij}=\begin{cases}
c'_{ij},& i+j<n,\\
\theta_i, & i+j=n,\\
0,& {\rm otherwise}.
\end{cases}$$

Suppose that $\varphi:\frakp(C)_\gamma\to \frakp(C')_\theta$ is a graded algebra isomorphism. By Theorem~\ref{th:GAA_iso}, there exists a tuple $(b_{i})_{i\in [n]}$ of nonzero coefficients such that $\varphi(p_i)=b_ip_i$ and
$$b_{i+j}d_{ij}=b_{i}b_{j}d'_{ij},\quad\forall i,j\geq1.$$
In particular, the restriction of $\varphi$ on $\frakp(C)$ is an isomorphism from $\frakp(C)$ to $\frakp(C')$. We set $\sigma=\varphi|_{\frakp(C)}$. Then for any $i,j\geq1$,
 \begin{eqnarray*}
 \theta_\sigma(p_i,p_j)&\stackrel{\eqref{eq:cocycle_action}}{=}&\theta(\sigma(p_i),\sigma(p_j))\\
 &=&b_{i}b_{j}\theta(p_i,p_j)\\
 &\stackrel{\eqref{eq:cocycle}}{=}&\delta_{i+j,\,n}b_{i}b_{j}d'_{ij}p_n\\
 &=&\delta_{i+j,\,n}b_{i+j}d_{ij}p_n\\
 &\stackrel{\eqref{eq:cocycle}}{=}&b_n\gamma(p_i,p_j).
 \end{eqnarray*}
That is, $\theta_\sigma=b_n\gamma$. Conversely, if there exist a graded algebra isomorphism $\sigma:\frakp(C)\to\frakp(C')$ and $b\in\mathbbm k^*$ such that $\theta_\sigma=b\gamma$, one can define the graded linear isomorphism $\varphi:\frakp(C)_\gamma\to \frakp(C')_\theta$ by letting $\varphi|_{\frakp(C)}=\sigma$ and $\varphi(p_n)=bp_n$. Indeed, $\varphi$ is clearly bijective. Moreover, for any $ i,j\in[n-1]$, if $i+j<n$, then
$$\varphi(p_i\cdot_\gamma p_j) \stackrel{\eqref{eq:central_extension}}{=} \varphi(p_ip_j)=\sigma(p_ip_j)=\sigma(p_i)\sigma(p_j)=\varphi(p_i) \varphi(p_j) \stackrel{\eqref{eq:central_extension}}{=}\varphi(p_i)\cdot_\theta \varphi(p_j).$$
Otherwise, if $i+j= n$, then
$$\varphi(p_i\cdot_\gamma p_j) \stackrel{\eqref{eq:central_extension}}{=} \varphi(\gamma(p_i,p_j)) \stackrel{\eqref{eq:cocycle}}{=} b\gamma(p_i,p_j)=\theta_\sigma(p_i,p_j),$$
and
$$\varphi(p_i)\cdot_\theta \varphi(p_j) \stackrel{\eqref{eq:central_extension}}{=}  \theta(\varphi(p_i),\varphi(p_j))
=\theta(\sigma(p_i),\sigma(p_j))\stackrel{\eqref{eq:cocycle_action}}{=}\theta_\sigma(p_i,p_j),$$
so it is easy to see that $\varphi$ is an algebra map.
\end{proof}

\begin{coro}\label{cor:stab}
Given any $C\in\mathcal M_{n-1}$ and $2$-cocycle $\theta\in Z_g^2(\frakp(C),V)$, we have $\frakp(C)_{\theta_{\varphi_b}}\cong \frakp(C)_\theta$ for any $\varphi_b=(b^i)_{i\in [n-1]}\in \Aut(\frakp(C))$ with $b\in \mathbbm k^*$.
\end{coro}
\begin{proof}
By Theorem~\ref{th:isom}, we only need to show that $\theta_{\varphi_b}$ and $\theta$ are colinear.
In fact,
$$\theta_{\varphi_b}(p_i,p_j)\stackrel{\eqref{eq:cocycle_action}}{=}\theta(\varphi_b(p_i),\varphi_b(p_j))
=b^{i+j}\theta(p_i,p_j)\stackrel{\eqref{eq:cocycle}}{=}b^n\theta(p_i,p_j),$$
so $\theta_{\varphi_b}=b^{n}\theta$.
Furthermore, given any $b'\in\mathbbm k^*$, we can actually take $b\in\mathbbm k^*$ such that $b^n=b'$, as $\mathbbm k$ is algebraically closed. Then $\theta_{\varphi_b}=b^n\theta=b'\theta$.
\end{proof}

\noindent
{\bf Convention.} Using the notation in Eq.~\eqref{eq:cocycle} we can identify any $2$-cocycle $\theta\in Z_g^2(\frakp(C),V)$
with a tuple $(\theta_i)_{i\in [n-1]}\in \mathbbm k^{n-1}$ subjecting to the following binomial relations, \begin{equation}\label{eq:cocycle_theta}
\theta_i=\theta_{n-i},\quad\theta_{i+j}c_{ij}=\theta_{i}c_{j,\,n-i-j}
\end{equation}
for any $i,j\geq 1$ satisfying $i+j\leq n-1$.
So
$$Z_g^2(\frakp(C),V)=\left\{(\theta_i)_{i\in [n-1]}\in \mathbbm k^{n-1}\,\big|\,
\theta_i=\theta_{n-i},\ \theta_{i+j}c_{ij}=\theta_{j+k}c_{jk},\ \forall i,j,k\geq1,\ i+j+k=n\right\}$$
is a subvariety of the affine space $\mathbbm k^{n-1}$.

\begin{prop}\label{prop:aut_stab}
Given any $C\in\mathcal M_{n-1}$ and nonzero $2$-cocycle $\theta$ in $Z_g^2(\frakp(C),V)$, the stabilizer of the $1$-dimensional subspace $[\theta]$ of $Z_g^2(\frakp(C),V)$ under the action of $\Aut(\frakp(C))$ is given by
\begin{equation}\label{eq:aut_stab}
{\rm Stab}_{\Aut(\frakp(C))}[\theta]=\{(b_i)_{i\in [n-1]}\in \Aut(\frakp(C))\,|\, b_ib_{n-i}=b_jb_{n-j}\ \mbox{ if }\ \theta_i\theta_j\neq 0\},
\end{equation}
\end{prop}
\begin{proof}
For any $\sigma=(b_i)_{i\in [n-1]}\in \Aut(\frakp(C))$, we know that
\begin{eqnarray*}
[\theta]\cdot\sigma=[\theta] &\Leftrightarrow&  [\theta_\sigma]=[\theta]\\
&\Leftrightarrow& \exists\, b \in \mathbbm k^*\ \mbox{such that}\ \theta_\sigma=b\theta\\
&\stackrel{\eqref{eq:cocycle},\,\eqref{eq:cocycle_action}}{\Leftrightarrow}& \exists\, b \in \mathbbm k^*\ \mbox{such that}\
b_ib_{n-i}=b\ {\rm whenever}\ \theta_i\neq 0,
\end{eqnarray*}
which implies the desired description of ${\rm Stab}_{\Aut(\frakp(C))}[\theta]$.
\end{proof}

\subsection{The study of central extensions via the associated graphs}
Next we introduce a crucial combinatorial tool to study the central extensions of DGCAs.
\begin{defn}\label{defn:graph}
Given any $C\in\mathcal M_{n-1}$ with $n\geq 2$, we define a undirected simple graph ${\rm Gr}(C)$ with respect to $C$ as follows. Its vertex set is $[n-1]$, and any two distinct vertices $p,q$ are adjacent
if and only if $p+q=n$ or
there exist $i,j,k\geq 1$ such that
\begin{equation}\label{eq:graph}
i+j+k=n,\ p=i+j,\ q=j+k,\ c_{ij}c_{jk}\neq 0.
\end{equation}
A connected component of ${\rm Gr}(C)$ is called {\bf contractible} if there exist $i,j\geq 1$ such that this connected component contains the vertex $i+j$, and $c_{ij}\neq0$ but $c_{j,\,n-i-j}=0$. Otherwise, we say that it is {\bf generic}.
For any $p\in [n-1]$, let ${\rm Gr}^{[p]}(C)$ be the connected component of ${\rm Gr}(C)$
containing vertex $p$.

\end{defn}

\begin{remark}\label{rem:graph_triple}
For any $C\in\mathcal M_{n-1}$ with $n\geq 2$, two vertices $p,q\in [n-1]$ are adjacent in ${\rm Gr}(C)$ only if $p+q\geq n$. When $p$ and $q$ are adjacent and $p+q>n$, we uniquely have
\begin{equation}\label{eq:graph_triple}
i=n-q,\ j=p+q-n,\ k=n-p
\end{equation}
satisfying Eq.~\eqref{eq:graph}.
\end{remark}

\begin{exam}For $n=5$,
let $C=\begin{pmatrix}
0&1&0&0\\
1&1&0&0\\
0&0&0&0\\
0&0&0&0\\
\end{pmatrix},
\ C'=\begin{pmatrix}
0&0&1&0\\
0&1&0&0\\
1&0&0&0\\
0&0&0&0\\
\end{pmatrix}\in\mathcal M_4$, then
$c_{12}c_{22}=1$, $c'_{22}=1$ but $c'_{12}=0$,  and $c'_{13}=1$ but $c'_{11}=0$,
so the graphs ${\rm Gr}(C)$ and ${\rm Gr}(C')$ are respectively
$$\raisebox{3em}{\small 
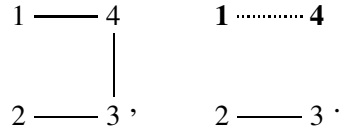
\xymatrix{1 \ar@{-}[r]& 4 \ar@{-}[d]\\
2\ar@{-}[r]& 3}},\qquad
\raisebox{3em}{\small \xymatrix{{\bf 1} \ar@{.}[r]& {\bf 4} \\
2\ar@{-}[r]& 3}}.$$
Note that ${\rm Gr}(C)$ is already connected and generic, while ${\rm Gr}(C')$ has two connected components, one of which containing vertices $1,4$ is contractible. Here we use numbers in bold and dotted lines to represent vertices and edges in a contractible connected component respectively.
\end{exam}

Now we state the main result of this section, an explicit characterization of the second (graded) commutative cohomology group of DGCAs.
\begin{theorem}\label{th:cocyle_graph}
For any $C\in\mathcal M_{n-1}$, the $2$-cocycle space $Z_g^2(\frakp(C),V)$ is a subvariety of the affine space $\mathbbm k^{n-1}$
of dimension not bigger than the number $u(C)$ of generic connected components of the graph ${\rm Gr}(C)$.
\end{theorem}
\begin{proof}
Fix any $2$-cocycle $\theta=(\theta_i)_{i\in [n-1]}\in Z_g^2(\frakp(C),V)$. According to Remark~\ref{rem:graph_triple}, we have $p+q\geq n$ for any two adjacent vertices $p$ and $q$ in ${\rm Gr}(C)$. If $p+q=n$, then $\theta_p=\theta_q$. If $p+q>n$, we have $c_{ij}c_{j,\,n-i-j}\neq0$ for $i=n-q$ and $j=p+q-n$ by Definition~\ref{defn:graph}, and also $c_{ij}\theta_p=c_{j,\,n-i-j}\theta_q$ by Eq.~\eqref{eq:cocycle_theta}.
In other words,
the values of $\theta$ on any connected component ${\rm Gr}^{[p]}(C)$ with $p\in [n-1]$ is uniquely determined by $\theta_p$.

From the discussion above, we know that if $\theta$ takes zero on one vertex, it must vanish on the entire connected component containing such a vertex.
Particularly, any contractible connected component of ${\rm Gr}(C)$ contains a vertex $p$ such that $i+j=p$, and $c_{ij}\neq0$ but $c_{j,\,n-i-j}=0$, then $\theta_p=\theta_{n-p}=0$
by Eq.~\eqref{eq:cocycle_theta} again. Hence, $\theta$ vanish on all contractible connected components.

As a result, $Z_g^2(\frakp(C),V)$ is a subvariety of $\mathbbm k^{n-1}$ isomorphic to $\mathbbm k^u$ with $u\leq u(C)$.
\end{proof}

\begin{remark}
(1) Following the proof of Theorem~\ref{th:cocyle_graph}, we see that the $2$-cocycle space $Z_g^2(\frakp(C),V)$ is exactly isomorphic to $\mathbbm k^{u(C)}$, when
all $u(C)$ generic connected components of the graph ${\rm Gr}(C)$ are circuit-free, namely they are trees.

(2) By contrast, a generic connected component of ${\rm Gr}(C)$ containing a circuit may force all $2$-cocycles vanishing on it, and $Z_g^2(\frakp(C),V)$ is of dimension $<u(C)$ in this situation. Such a phenomenon drives us to define a more subtle notion. That is, a generic connected component is called {\bf nonvanishing} if it admits nonzero values of 2-cocyles.

(3) For any $(0,1)$-matrix $C$ in $\mathcal M_{n-1}$ and $2$-cocycle $\theta\in Z_g^2(\frakp(C),V)$, $\theta$ is a constant function on every connected component of ${\rm Gr}(C)$, as all nonzero entries of $C$ are 1. In particular, any generic connected component of ${\rm Gr}(C)$ is {\it always} nonvanishing when $C$ is a $(0,1)$-matrix.
\end{remark}

\begin{exam}
Let $C=(c_{ij})_{i,j\in [10]}$ be any $10\times 10$ symmetric matrix with nonzero entries
$$c_{24}=c_{42},\ c_{33},\ c_{25}=c_{52},\ c_{34}=c_{43},\ c_{35}=c_{53},\ c_{44},\ c_{45}=c_{54},$$
then Eq.~\eqref{eq:ass_coef} trivially holds for such a matrix and so $C\in\mathcal M_{10}$. Since
$$c_{24}c_{25},\ c_{24}c_{45},\ c_{25}c_{45},\ c_{33}c_{35},\ c_{34}c_{44}\neq0,$$
the graph ${\rm Gr}(C)$ is given by
$$
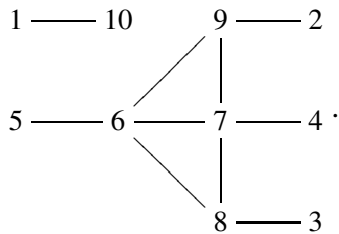
\raisebox{3em}{\small \xymatrix{
1\ar@{-}[r] & 10 & 9\ar@{-}[r]& 2\\
5 \ar@{-}[r]& 6 \ar@{-}[r] \ar@{-}[ru] \ar@{-}[rd] & 7 \ar@{-}[r] \ar@{-}[u] \ar@{-}[d] & 4\\
&  & 8\ar@{-}[r] & 3}}.$$
Note that there is a circuit with vertices $6,7,8$. If $\theta=(\theta_i)_{i\in[10]}$ is a $2$-cocycle in $Z_g^2(\frakp(C),V)$, then Eq.~\eqref{eq:cocycle_theta} implies that
$$c_{24}\theta_6=c_{52}\theta_7,\ c_{34}\theta_7=c_{44}\theta_8,\ c_{35}\theta_8=c_{33}\theta_6.$$
Hence, whenever the following equality
$$c_{24}c_{34}c_{35}=c_{52}c_{44}c_{33}$$
fails, we have $\theta_6\theta_7\theta_8=0$, and then $\theta$ vanishes on the whole connected component containing this circuit. So even though ${\rm Gr}(C)$ has 2 generic connected components, the space $Z_g^2(\frakp(C),V)$ may only have dimension 1.

\end{exam}

Also, we introduce the following equivalence relation on the $2$-cocycle space of a DGCA.
\begin{defn}
Given any $C\in\mathcal M_{n-1}$, the equivalence relation $\sim$ on $\mathcal M_{n-1}$ induces an equivalence relation, also denoted by $\sim$, on $Z_g^2(\frakp(C),V)$. That is, for any $\theta,\gamma\in Z_g^2(\frakp(C),V)$, let
$$\theta \sim \gamma \ \Leftrightarrow \ C_\theta\sim C_\gamma,$$
where $C_\theta$ and $C_\gamma$ are respectively the coefficient matrices of $\frakp(C)_\theta$ and $\frakp(C)_\gamma$ as $n$-dimensional DGCAs.
Also, let $Z_\gamma$ be the equivalence class of $\gamma$ in $Z_g^2(\frakp(C),V)$ with respect to $\sim$.
\end{defn}
Namely, $\theta\in Z_\gamma$ if and only if $\theta,\gamma$ have the same distribution of nonzero entries.

\begin{coro}\label{cor:cocyle_graph}
For any $C\in\mathcal M_{n-1}$ and $\gamma\in Z_g^2(\frakp(C),V)$, the equivalence class $Z_\gamma$ is a locally closed subvariety of $\mathbbm k^{n-1}$ isomorphic to $(\mathbbm k^*)^{\times v(\gamma)}$, where
$v(\gamma)$ is the number of connected components of ${\rm Gr}(C)$ on which $\gamma$ takes nonzero values.
\end{coro}
\begin{proof}
First arbitrarily choose one vertex from each generic connected component of ${\rm Gr}(C)$ to form a $u(C)$-subset $X$ of $[n-1]$.
According to Theorem~\ref{th:cocyle_graph},
any $2$-cocycle $\theta\in Z_g^2(\frakp(C),V)$ can be uniquely determined by any given values $\theta_p\ (p\in X)$, while other $\theta_q$ must vanish whenever ${\rm Gr}^{[q]}(C)$ is contractible. Especially if $\theta\in Z_\gamma$, then we also have $\theta_p\neq0$ for exactly $v(\gamma)$ vertices $p\in X$ such that $\gamma_p\neq0$, hence prove the result.
\end{proof}

Next we give a necessary condition when the central extensions arisen from one equivalence class $Z_\gamma$ in $Z_g^2(\frakp(C),V)$ only give a unique isomorphism class of DGCAs.
\begin{prop}\label{prop:orbit}
For any matrix $C\in\mathcal M_{n-1}$ and nonzero $2$-cocycle $\gamma\in Z_g^2(\frakp(C),V)$, if
the central extensions
$\frakp(C)_\theta$ with $\theta\in Z_\gamma$ form a unique isomorphism class,
then the dimension of the $\Aut(\frakp(C))$-orbit $O_{[\gamma]}$ of the $1$-dimensional subspace $[\gamma]$, or equivalently
the number of non-redundant equalities defining ${\rm Stab}_{\Aut(\frakp(C))}[\gamma]$ in Eq.~\eqref{eq:aut_stab}, is exactly $v(\gamma)-1$.
\end{prop}
\begin{proof}
By \cite[Prop.~1.11]{Br} we know that the orbit $O_{[\gamma]}$ is a locally closed subvariety of the Grassmannian of $1$-dimensional subspaces of $Z_g^2(\frakp(C),V)$ and
$$\dim \mathcal O_{[\gamma]}=\dim \Aut(\frakp(C)) - \dim{\rm Stab}_{\Aut(\frakp(C))}[\gamma],$$
which is consequently the number of non-redundant equalities defining ${\rm Stab}_{\Aut(\frakp(C))}[\gamma]$ in Eq.~\eqref{eq:aut_stab}, under the premise of those defining $\Aut(\frakp(C))$ in Eq.~\eqref{eq:aut}.

On the other hand,  since $Z_\gamma$ contains the $\Aut(\frakp(C))$-orbit $\mathcal O_\gamma$ in $Z_g^2(\frakp(C),V)$, we clearly have $\dim \mathcal O_{[\gamma]}<\dim Z_\gamma=v(\gamma)$ by Corollary~\ref{cor:cocyle_graph}. Such an inequality can also be seen as follows. For any $\sigma=(b_i)_{i\in [n-1]}\in \Aut(\frakp(C))$, we claim that
$b_pb_{n-p}=b_qb_{n-q}$
if vertices $p$ and $q$ are connected in ${\rm Gr}(C)$. In fact, by the connectedness we only need to check it
whenever $p$ and $q$ are adjacent. Now it is trivially true if $p+q=n$. Otherwise, by Definition~\ref{defn:graph} there exist $i,j,k\geq 1$ such that
$$i+j+k=n,\ p=i+j,\ q=j+k,\ c_{ij}c_{jk}\neq 0,$$
and then
\begin{align*}
b_pb_{n-p}&=b_{i+j}b_{n-i-j}
\stackrel{\eqref{eq:aut}}{=}b_ib_jb_k,\\
b_qb_{n-q}&=b_{j+k}b_{n-j-k}
\stackrel{\eqref{eq:aut}}{=}b_ib_jb_k,
\end{align*}
so $b_pb_{n-p}=b_qb_{n-q}$.
As a result, the number of non-redundant equalities defining ${\rm Stab}_{\Aut(\frakp(C))}[\gamma]$ in Eq.~\eqref{eq:aut_stab} is less than the number $v(\gamma)$ of connected components of ${\rm Gr}(C)$ on which $\gamma$ does not vanish.

According to Theorem~\ref{th:isom},
the isomorphism classes of central extensions of $\frakp(C)$ are in bijection with the $\Aut(\frakp(C))$-orbits in the Grassmannian of $1$-dimensional subspaces of $Z_g^2(\frakp(C),V)$. If $\dim \mathcal O_{[\gamma]}<v(\gamma)-1$, then $Z_\gamma$ contains other $\Aut(\frakp(C))$-orbits apart from $\mathcal O_{[\gamma]}$,
and the central extensions $\frakp(C)_\theta$ with $\theta\in Z_\gamma$, i.e. $C_\theta\sim C_\gamma$ provide more than one isomorphism class.
\end{proof}

Next we use the combinatorial tool just developed to explain the phenomenon appearing in Example~\ref{ex:counter}.
\begin{exam}\label{ex:counter'}
Let $C=(c_{ij})_{i,j\in [7]}$ be the symmetric matrix in $\mathcal M_7$ with
$$c_{24}=c_{42}=c_{33}=c_{25}=c_{52}=c_{34}=c_{43}=1$$
and zeros elsewhere. Let $\gamma=(\gamma_i)_{i\in[7]}$ be the $2$-cocycle in $Z_g^2(\frakp(C),V)$ with
$$\gamma_1=\gamma_2=\gamma_6=\gamma_7=0,\quad \gamma_3=\gamma_4=\gamma_5=1.$$
By Proposition~\ref{prop:aut},
$$\Aut(\frakp(C))=\left\{(b_i)_{i\in [7]}\in(\mathbbm k^*)^{\times 7}\,\big|\, b_6=b_2b_4=b_3^2,\ b_7=b_2b_5=b_3b_4\right\}.$$
For any $\sigma=(b_i)_{i\in [7]}\in \Aut(\frakp(C))$, we have
$$b_3b_4^2=b_2b_4b_5=b_3^2b_5,\quad\mbox{i.e.}\quad b_4^2=b_3b_5.$$
Using Eq.~\eqref{eq:aut_stab}, we know that $\sigma\in {\rm Stab}_{\Aut(\frakp(C))}[\theta]$ for any $\theta\in Z_\gamma$, and ${\rm Stab}_{\Aut(\frakp(C))}[\theta]=\Aut(\frakp(C))$.
Hence, the $\Aut(\frakp(C))$-orbit of $[\theta]$ is the singleton $\{[\theta]\}$ (of dimension 0). On the other hand, the graph ${\rm Gr}(C)$ is given by
$$
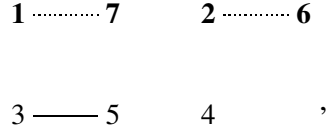
\raisebox{3em}{\small \xymatrix{{\bf 1} \ar@{.}[r]& {\bf 7} & {\bf 2} \ar@{.}[r]& {\bf 6}\\
3\ar@{-}[r]& 5 & 4 &}},$$
and Corollary~\ref{cor:cocyle_graph} implies that
$$Z_\gamma=\left\{\theta=(\theta_i)_{i\in [7]}\in \mathbbm k^7\,\big|\,
\theta_1=\theta_2=\theta_6=\theta_7=0,\ \theta_3=\theta_5\neq0,\ \theta_4\neq0\right\}$$
of dimension $v(\gamma)=2$, and the Grassmannian of $1$-dimensional subspaces of $Z_\gamma$ has dimension 1. According to Proposition~\ref{prop:orbit},
the central extensions $\frakp(C)_\theta$ with $\theta\in Z_\gamma$, i.e. $C_\theta\sim C_\gamma$ provide more than one isomorphism class of DGCAs of dimension 8,
and even infinitely many when $\mathbbm k$ is infinite.
\end{exam}

With evidences of small dimension, we also try to prove that any DGCA has nontrivial central extensions,
but still lack of proper methods. Maybe the general theory of binomial algebras is helpful for this problem~\cite{ES}.
\begin{conjecture}
For any  $C\in\mathcal M_{n-1}$, ${\rm Gr}(C)$ must have at least one nonvanishing connected component,
so the DGCA $\frakp(C)$ has nontrivial central extensions by Theorem~\ref{th:cocyle_graph}.
\end{conjecture}

In the end of this section, we give a criterion when all nontrivial central extensions of a DGCA form exactly one isomorphism class. In order to prove it,
we need a technical lemma.
\begin{lemma}\label{lem:cocyle_colinear}
For any matrix $C\in\mathcal M_{n-1}$, let $G$ be one connected component of ${\rm Gr}(C)$, and $\theta,\gamma$ two $2$-cocycles in $Z_g^2(\frakp(C),V)$.
Then $\theta$ and $\gamma$ are colinear on $G$, namely
$$\theta_p\gamma_q=\theta_q\gamma_p,\quad\forall p,q \in G.$$
\end{lemma}
\begin{proof}
When $G$ is contractible, any $2$-cocycle vanishes on $G$, so it is trivially true.

Otherwise, if $G$ is generic,
by the connectedness of $G$ it is sufficient to check that
$\theta_p\gamma_q=\theta_q\gamma_p$,
whenever $p$ and $q$ are adjacent in $G$. It is automatically true if $p+q=n$. Otherwise, by Definition~\ref{defn:graph} there exist $i,j,k\geq 1$ such that
$$i+j+k=n,\ p=i+j,\ q=j+k,\ c_{ij}c_{jk}\neq 0,$$
then
$$\theta_p\gamma_q=\theta_{i+j}\gamma_{j+k}\quad\mbox{and}\quad
\theta_q\gamma_p=\theta_{j+k}\gamma_{i+j}$$
are equal to each other as $c_{ij}\theta_{i+j}c_{jk}\gamma_{j+k}=c_{jk}\theta_{j+k}c_{ij}\gamma_{i+j}$ by Eq.~\eqref{eq:cocycle_theta}.
\end{proof}

\begin{theorem}\label{th:orbit}
For any matrix $C\in\mathcal M_{n-1}$, all nontrivial central extensions of $\frakp(C)$ form exactly one isomorphism class of DGCAs if and only if ${\rm Gr}(C)$ has exactly one nonvanishing connected component.
\end{theorem}
\begin{proof}
As ${\rm Gr}(C)$ has only one nonvanishing connected component, we know that
any two nonzero $2$-cocycles $\theta,\gamma\in Z_g^2(\frakp(C),V)$ must take nonzero values on it, and are already colinear by Lemma~\ref{lem:cocyle_colinear}.  So $\frakp(C)_\theta\cong \frakp(C)_\gamma$ by Theorem~\ref{th:isom}, and all nontrivial central extensions of $\frakp(C)$ give exactly one isomorphism class of DGCAs.

Conversely, suppose that all nontrivial central extensions of $\frakp(C)$ form one isomorphism class of DGCAs. Then for any two nonzero $2$-cocycles $\theta,\gamma\in Z_g^2(\frakp(C),V)$, we have $\frakp(C)_\theta\cong \frakp(C)_\gamma$, which is equivalent to the existence of $\sigma\in \Aut(\frakp(C))$ making $\theta_\sigma$ and $\gamma$ colinear by Theorem~\ref{th:isom}. In other words, there exist nonzero $2$-cocycles and all of them have the same distribution of nonzero entries, which forces that ${\rm Gr}(C)$ has exactly one nonvanishing connected component by Corollary~\ref{cor:cocyle_graph}.
\end{proof}

\section{Classification of diagonally graded commutative algebras of low dimension}\label{sec:low_dim}
In this section, we list the classification results of diagonally graded commutative algebras of low dimension as an illustration.

Let $\frakp(C)$ be an $n$-dimensional DGCA associated with the coefficient matrix $C$ in $\mathcal M_n$. The Classification Theorem~\ref{th:classification} says that the isomorphism classes of DGCAs of dimension $\leq 7$ are determined by the distribution of nonzero entries in coefficient matrices, so they can be classified by listing the $(0,1)$-coefficient matrices of the canonical representative system of isomorphism classes.
For the cases when $n=2,3,4$, we successively have
\begin{center}
\begin{table}[h]
\renewcommand{\arraystretch}{.92}
\centering
\resizebox{\linewidth}{!}{
\begin{tabular}{|cccccc|}
\hline
\multicolumn{2}{|c|}{$n=2$} & \multicolumn{4}{c|}{$n=3$} \\ \hline
\multicolumn{2}{|c|}{\vspace{-.9em}}& \multicolumn{4}{c|}{}\\
\multicolumn{1}{|c}{\hspace{1.1em} $\begin{pmatrix}
0 & 0 \\
0 & 0 \\
\end{pmatrix}$} & \multicolumn{1}{c|}{\hspace{-.9em} $\begin{pmatrix}
1 & 0 \\
0 & 0 \\
\end{pmatrix}$ } & \multicolumn{1}{c}{$\begin{pmatrix}
0 & 0 & 0 \\
0 & 0 & 0 \\
0 & 0 & 0 \\
\end{pmatrix}$} & \multicolumn{1}{c}{$\begin{pmatrix}
1 & 0 & 0 \\
0 & 0 & 0 \\
0 & 0 & 0 \\
\end{pmatrix}$} & \multicolumn{1}{c}{$\begin{pmatrix}
0 & 1 & 0 \\
1 & 0 & 0 \\
0 & 0 & 0 \\
\end{pmatrix}$} & $\begin{pmatrix}
1 & 1 & 0 \\
1 & 0 & 0 \\
0 & 0 & 0 \\
\end{pmatrix}$\\
\hline
\multicolumn{1}{|c|}{\vspace{-.9em}}& \multicolumn{5}{c|}{}\\
\multicolumn{1}{|c|}{\multirow{5}{*}{$n=4$}} & $\begin{pmatrix}
0&0&0&0\\
0&0&0&0\\
0&0&0&0\\
0&0&0&0\\
\end{pmatrix}$ & \multicolumn{1}{c}{$\begin{pmatrix}
1&0&0&0\\
0&0&0&0\\
0&0&0&0\\
0&0&0&0\\
\end{pmatrix}$} & \multicolumn{1}{c}{$\begin{pmatrix}
0&0&0&0\\
0&1&0&0\\
0&0&0&0\\
0&0&0&0\\
\end{pmatrix}$} & \multicolumn{1}{c}{$\begin{pmatrix}
0&1&0&0\\
1&0&0&0\\
0&0&0&0\\
0&0&0&0\\
\end{pmatrix}$} & $\begin{pmatrix}
0&0&1&0\\
0&0&0&0\\
1&0&0&0\\
0&0&0&0\\
\end{pmatrix}$\\ 
\multicolumn{1}{|c|}{\vspace{-.8em}}& \multicolumn{5}{c|}{}\\
\multicolumn{1}{|c|}{} & $\begin{pmatrix}
1&1&0&0\\
1&0&0&0\\
0&0&0&0\\
0&0&0&0\\
\end{pmatrix}$ & \multicolumn{1}{c}{$\begin{pmatrix}
1&0&1&0\\
0&0&0&0\\
1&0&0&0\\
0&0&0&0\\
\end{pmatrix}$} & \multicolumn{1}{c}{$\begin{pmatrix}
0&1&0&0\\
1&1&0&0\\
0&0&0&0\\
0&0&0&0\\
\end{pmatrix}$} & \multicolumn{1}{c}{$\begin{pmatrix}
0&0&1&0\\
0&1&0&0\\
1&0&0&0\\
0&0&0&0\\
\end{pmatrix}$} & $\begin{pmatrix}
1&1&1&0\\
1&1&0&0\\
1&0&0&0\\
0&0&0&0\\
\end{pmatrix}$\\[.2em]
 \hline
\end{tabular}}
\end{table}
\end{center}
\vspace{-.5cm}

For the case when $n=5$, we list the $(0,1)$-coefficient matrices $C$ of the canonical representatives of 22 isomorphism classes as follows, for each of which one concrete isomorphism $(b_{i})_{1\leq i\leq5}$ between $\calp(C)$ and $\calp(C')$ is also given whenever $C'\sim C$.
\newcounter{rowno}
\begin{center}
\renewcommand{\arraystretch}{1.7}
\begin{longtable*}{|m{1.5cm}<{\centering}|m{4cm}<{\centering}|m{8cm}<{\centering}|}
\hline
No. & $C$ & $(b_{i})_{1\leq i\leq5}$
\\ \endhead
\hline
\stepcounter{rowno}\therowno&
\renewcommand{\arraystretch}{.97}
$\left(\begin{array}{lllll}
0 & 0 & 0 & 0 & 0 \\
0 & 0 & 0 & 0 & 0 \\
0 & 0 & 0 & 0 & 0 \\
0 & 0 & 0 & 0 & 0 \\
0 & 0 & 0 & 0 & 0
\end{array}\right)$ & $b_{i}\in\mathbbm k^*$ for $1\leq i\leq 5$ \\
\hline
\stepcounter{rowno}\therowno&
\renewcommand{\arraystretch}{.97}
$\left(\begin{array}{lllll}
1 & 0 & 0 & 0 & 0 \\
0 & 0 & 0 & 0 & 0 \\
0 & 0 & 0 & 0 & 0 \\
0 & 0 & 0 & 0 & 0 \\
0 & 0 & 0 & 0 & 0
\end{array}\right)$ & \makecell{$b_{1}=1,\,b_{2}=c_{11}'$,\\[.7em]
$b_{3},\,b_{4},\,b_{5}\in\mathbbm k^*$} \\
\hline
\stepcounter{rowno}\therowno&
\renewcommand{\arraystretch}{.97}
$\left(\begin{array}{lllll}
0 & 0 & 0 & 0 & 0 \\
0 & 1 & 0 & 0 & 0 \\
0 & 0 & 0 & 0 & 0 \\
0 & 0 & 0 & 0 & 0 \\
0 & 0 & 0 & 0 & 0
\end{array}\right)$ & \makecell{$b_{2}=1,\,b_{4}=c_{22}'$,\\[.7em]
$b_{1},\,b_{3},\,b_{5}\in\mathbbm k^*$} \\
\hline
\stepcounter{rowno}\therowno&
\renewcommand{\arraystretch}{.97}
$\left(\begin{array}{lllll}
0 & 1 & 0 & 0 & 0 \\
1 & 0 & 0 & 0 & 0 \\
0 & 0 & 0 & 0 & 0 \\
0 & 0 & 0 & 0 & 0 \\
0 & 0 & 0 & 0 & 0
\end{array}\right)$ & \makecell{$b_{1}=b_{2}=1,\,b_{3}=c_{12}'$,\\[.7em]
$b_{4},\,b_{5}\in\mathbbm k^*$} \\
\hline
\stepcounter{rowno}\therowno&
\renewcommand{\arraystretch}{.97}
$\left(\begin{array}{lllll}
0 & 0 & 1 & 0 & 0 \\
0 & 0 & 0 & 0 & 0 \\
1 & 0 & 0 & 0 & 0 \\
0 & 0 & 0 & 0 & 0 \\
0 & 0 & 0 & 0 & 0
\end{array}\right)$ & \makecell{$b_{1}=b_{3}=1,\,b_{4}=c_{13}'$,\\[.7em]
$b_{2},\,b_{5}\in\mathbbm k^*$}\\
\hline
\stepcounter{rowno}\therowno&
\renewcommand{\arraystretch}{.97}
$\left(\begin{array}{lllll}
0 & 0 & 0 & 0 & 0 \\
0 & 0 & 1 & 0 & 0 \\
0 & 1 & 0 & 0 & 0 \\
0 & 0 & 0 & 0 & 0 \\
0 & 0 & 0 & 0 & 0
\end{array}\right)$ & \makecell{$b_2=b_3=1,\,b_5=c_{23}'$,\\[.7em]
$b_1,\,b_4\in\mathbbm k^*$} \\
\hline
\stepcounter{rowno}\therowno&
\renewcommand{\arraystretch}{.97}
$\left(\begin{array}{lllll}
0 & 0 & 0 & 1 & 0 \\
0 & 0 & 0 & 0 & 0 \\
0 & 0 & 0 & 0 & 0 \\
1 & 0 & 0 & 0 & 0 \\
0 & 0 & 0 & 0 & 0
\end{array}\right)$ & \makecell{$b_{1}=b_{4}=1,\,b_{5}=c_{14}'$,\\[.7em]
$b_{2},\,b_{3}\in\mathbbm k^*$} \\
\hline
\stepcounter{rowno}\therowno&
\renewcommand{\arraystretch}{.97}
$\left(\begin{array}{lllll}
1 & 1 & 0 & 0 & 0 \\
1 & 0 & 0 & 0 & 0 \\
0 & 0 & 0 & 0 & 0 \\
0 & 0 & 0 & 0 & 0 \\
0 & 0 & 0 & 0 & 0
\end{array}\right)$ & \makecell{$b_{1}=1,\,b_{2}=c_{11}',\,b_{3}=c_{11}'c_{12}'$,\\[.7em]
$b_{4},\,b_{5}\in\mathbbm k^*$} \\
\hline
\stepcounter{rowno}\therowno&
\renewcommand{\arraystretch}{.97}
$\left(\begin{array}{lllll}
1 & 0 & 1 & 0 & 0 \\
0 & 0 & 0 & 0 & 0 \\
1 & 0 & 0 & 0 & 0 \\
0 & 0 & 0 & 0 & 0 \\
0 & 0 & 0 & 0 & 0
\end{array}\right)$ & \makecell{$b_{1}=b_{3}=1,\,b_{2}=c_{11}',\,b_{4}=c_{13}'$,\\[.7em]
$b_{5}\in\mathbbm k^*$} \\
\hline
\stepcounter{rowno}\therowno&
\renewcommand{\arraystretch}{.97}
$\left(\begin{array}{lllll}
1 & 0 & 0 & 1 & 0 \\
0 & 0 & 0 & 0 & 0 \\
0 & 0 & 0 & 0 & 0 \\
1 & 0 & 0 & 0 & 0 \\
0 & 0 & 0 & 0 & 0
\end{array}\right)$ & \makecell{$b_{1}=b_{4}=1,\,b_{2}=c_{11}',\,b_{5}=c_{14}'$,\\[.7em]
$b_{3}\in\mathbbm k^*$} \\
\hline
\stepcounter{rowno}\therowno&
\renewcommand{\arraystretch}{.97}
$\left(\begin{array}{lllll}
0 & 1 & 0 & 0 & 0 \\
1 & 1 & 0 & 0 & 0 \\
0 & 0 & 0 & 0 & 0 \\
0 & 0 & 0 & 0 & 0 \\
0 & 0 & 0 & 0 & 0
\end{array}\right)$ & \makecell{$b_{1}=b_{2}=1,\,b_{3}=c_{12}',\,b_{4}=c_{22}'$,\\[.7em]
$b_{5}\in\mathbbm k^*$} \\
\hline
\stepcounter{rowno}\therowno&
\renewcommand{\arraystretch}{.97}
$\left(\begin{array}{lllll}
0 & 0 & 1 & 0 & 0 \\
0 & 1 & 0 & 0 & 0 \\
1 & 0 & 0 & 0 & 0 \\
0 & 0 & 0 & 0 & 0 \\
0 & 0 & 0 & 0 & 0
\end{array}\right)$ & \makecell{$b_{1}=b_{2}=1,\,b_{3}=\dfrac{c_{22}'}{c_{13}'},\,b_{4}=c_{22}'$,\\[1em]
$b_{5}\in\mathbbm k^*$} \\
\hline
\stepcounter{rowno}\therowno&
\renewcommand{\arraystretch}{.97}
$\left(\begin{array}{lllll}
0 & 0 & 0 & 0 & 0 \\
0 & 1 & 1 & 0 & 0 \\
0 & 1 & 0 & 0 & 0 \\
0 & 0 & 0 & 0 & 0 \\
0 & 0 & 0 & 0 & 0
\end{array}\right)$ & \makecell{$b_{2}=b_{3}=1,\,b_{4}=c_{22}',\,b_{5}=c_{23}'$,\\[.7em]
$b_{1}\in\mathbbm k^*$} \\
\hline
\stepcounter{rowno}\therowno&
\renewcommand{\arraystretch}{.97}
$\left(\begin{array}{lllll}
0 & 1 & 0 & 1 & 0 \\
1 & 0 & 0 & 0 & 0 \\
0 & 0 & 0 & 0 & 0 \\
1 & 0 & 0 & 0 & 0 \\
0 & 0 & 0 & 0 & 0
\end{array}\right)$ & $b_{1}=b_{2}=b_{4}=1,\,b_{3}=c_{12}',\,b_{5}=c_{14}'$ \\
\hline
\stepcounter{rowno}\therowno&
\renewcommand{\arraystretch}{.97}
$\left(\begin{array}{lllll}
0 & 0 & 1 & 0 & 0 \\
0 & 0 & 1 & 0 & 0 \\
1 & 1 & 0 & 0 & 0 \\
0 & 0 & 0 & 0 & 0 \\
0 & 0 & 0 & 0 & 0
\end{array}\right)$ & $b_{1}=b_{2}=b_{3}=1,\,b_{4}=a'_{13},\,b_{5}=c_{23}'$ \\
\hline
\stepcounter{rowno}\therowno&
\renewcommand{\arraystretch}{.97}
$\left(\begin{array}{lllll}
0 & 0 & 0 & 1 & 0 \\
0 & 0 & 1 & 0 & 0 \\
0 & 1 & 0 & 0 & 0 \\
1 & 0 & 0 & 0 & 0 \\
0 & 0 & 0 & 0 & 0
\end{array}\right)$ & $b_{1}=b_{2}=b_{5}=1,\,b_{3}=c_{23}'^{-1},\,b_{4}=c_{14}'^{-1}$ \\
\hline
\stepcounter{rowno}\therowno&
\renewcommand{\arraystretch}{.97}
$\left(\begin{array}{lllll}
1 & 1 & 0 & 1 & 0 \\
1 & 0 & 0 & 0 & 0 \\
0 & 0 & 0 & 0 & 0 \\
1 & 0 & 0 & 0 & 0 \\
0 & 0 & 0 & 0 & 0
\end{array}\right)$ & $b_{1}=b_{4}=1,\,b_{2}=c_{11}',\,b_{3}=c_{11}'c_{12}',\,b_{5}=c_{14}'$ \\
\hline
\stepcounter{rowno}\therowno&
\renewcommand{\arraystretch}{.97}
$\left(\begin{array}{lllll}
0 & 0 & 1 & 0 & 0 \\
0 & 1 & 1 & 0 & 0 \\
1 & 1 & 0 & 0 & 0 \\
0 & 0 & 0 & 0 & 0 \\
0 & 0 & 0 & 0 & 0
\end{array}\right)$ & \makecell{$b_{1}=b_{2}=1,\,b_{3}=\dfrac{c_{22}'}{c_{13}'},\,b_{4}=c_{22}'$,\\[1em] $b_{5}=\dfrac{c_{22}'c_{23}'}{c_{13}'}$} \\
\hline
\stepcounter{rowno}\therowno&
\renewcommand{\arraystretch}{.97}
$\left(\begin{array}{lllll}
1 & 1 & 1 & 0 & 0 \\
1 & 1 & 0 & 0 & 0 \\
1 & 0 & 0 & 0 & 0 \\
0 & 0 & 0 & 0 & 0 \\
0 & 0 & 0 & 0 & 0
\end{array}\right)$ & \makecell{$b_{1}=1,\,b_{2}=c_{11}',\,b_{3}=c_{11}'c_{12}'$,\\[.7em]
$b_{4}=c_{11}'^{2}c_{22}'=c_{11}'c_{12}'c_{13}'$,\\[.7em]
$b_{5}\in\mathbbm k^*$} \\
\hline
\stepcounter{rowno}\therowno&
\renewcommand{\arraystretch}{.97}
$\left(\begin{array}{lllll}
1 & 0 & 1 & 1 & 0 \\
0 & 0 & 1 & 0 & 0 \\
1 & 1 & 0 & 0 & 0 \\
1 & 0 & 0 & 0 & 0 \\
0 & 0 & 0 & 0 & 0
\end{array}\right)$ & \makecell{$b_{1}=b_{3}=1,\,b_{2}=c_{11}',\,b_{4}=c_{13}'$,\\[.8em]
$b_{5}=c_{11}'c_{23}'=c_{13}'c_{14}'$} \\
\hline
\stepcounter{rowno}\therowno&
\renewcommand{\arraystretch}{.97}
$\left(\begin{array}{lllll}
0 & 1 & 0 & 1 & 0 \\
1 & 1 & 1 & 0 & 0 \\
0 & 1 & 0 & 0 & 0 \\
1 & 0 & 0 & 0 & 0 \\
0 & 0 & 0 & 0 & 0
\end{array}\right)$ & \makecell{$b_{1}=b_{2}=1,\,b_{3}=c_{12}',\,b_{4}=c_{22}'$,\\[.8em]
$b_{5}=c_{12}'c_{23}'=c_{22}'c_{14}'$} \\
\hline
\stepcounter{rowno}\therowno&
\renewcommand{\arraystretch}{.97}
$\left(\begin{array}{lllll}
1 & 1 & 1 & 1 & 0 \\
1 & 1 & 1 & 0 & 0 \\
1 & 1 & 0 & 0 & 0 \\
1 & 0 & 0 & 0 & 0 \\
0 & 0 & 0 & 0 & 0
\end{array}\right)$ & \makecell{$b_{1}=1,\,b_{2}=c_{11}',\,b_{3}=c_{11}'c_{12}'$,\\[.7em]
$b_{4}=c_{11}'^{2}c_{22}'=c_{11}'c_{12}'c_{13}'$,\\[1em]
$b_{5}=c_{11}'^{2}c_{12}'c_{23}'=c_{11}'c_{12}'c_{13}'c_{14}'$}\\
\hline
\end{longtable*}
\end{center}

For the cases of higher dimension 6 and 7, it is similar to check but requiring much more effort, for their large numbers of isomorphism classes, and we will not present the whole classification result here. Alternatively, we choose to list two typical isomorphism classes of dimension 6 and 7 respectively by the same way as previous.
\begin{center}
\begin{tabular}{|m{1.5cm}<{\centering}|m{4.5cm}<{\centering}|m{8cm}<{\centering}|}
\hline
 & $C$ & $(b_{i})_{1\leq i\leq n}$ \\ \hline 
\makecell{$n=6$}
& \renewcommand{\arraystretch}{.97}
$\left(\begin{array}{llllll}
1 & 1 & 0 & 1 & 1 & 0 \\
1 & 0 & 0 & 1 & 0 & 0 \\
0 & 0 & 0 & 0 & 0 & 0 \\
1 & 1 & 0 & 0 & 0 & 0 \\
1 & 0 & 0 & 0 & 0 & 0 \\
0 & 0 & 0 & 0 & 0 & 0
\end{array}\right)$ & \makecell{$b_{1}=b_{4}=1,\,b_{2}=c_{11}',\,b_{3}=c_{11}'c_{12}',\,b_{5}=c_{14}'$,\\[.8em]
$b_{6}=c_{14}'c_{15}'=c_{11}'c_{24}'$} \\ \hline
%
\makecell{$n=7$} & \renewcommand{\arraystretch}{.97}
$\left(\begin{array}{lllllll}
1 & 0 & 1 & 0 & 1 & 1 & 0 \\
0 & 0 & 0 & 0 & 1 & 0 & 0 \\
1 & 0 & 0 & 0 & 0 & 0 & 0 \\
0 & 0 & 0 & 0 & 0 & 0 & 0 \\
1 & 1 & 0 & 0 & 0 & 0 & 0 \\
1 & 0 & 0 & 0 & 0 & 0 & 0 \\
0 & 0 & 0 & 0 & 0 & 0 & 0
\end{array}\right)$ & \makecell{$b_{1}=b_{3}=b_{5}=1,\,b_{2}=c_{11}',\,b_{4}=c_{13}',\,b_{6}=c_{15}'$,\\[.8em]
$b_{7}=c_{15}'c_{16}'=c_{11}'c_{25}'$} \\ \hline
\end{tabular}
\end{center}

\medskip
\begin{remark}
(1) The above classification result of DGCAs of low dimension is clearly available for arbitrary base field $\mathbbm k$. Particularly, the isomorphism classes of DGCAs over any prime field are already indivisible, and do not subdivide over arbitrary field extensions.

\smallskip
(2) The numbers of $(0,1)$-matrices in $\mathcal M_n\ (1\leq n\leq 10)$ are listed as follows.
\smallskip
\begin{center}
\begin{tabular}{|c|c|c|c|c|c|c|c|c|c|c|c|}
\hline
$n$ & 1 & 2 & 3 & 4 & 5 & 6 & 7 & 8 & 9 & 10 & $\cdots$\\
 \hline
$\#\ \mathcal M_n/\sim$ & 1 & 2 & 4 & 10 & 22 & 78 & 202 & 804 & 2824 & 14294 & $\cdots$\\
\hline
\end{tabular}
\end{center}
\smallskip
As we have seen, such an integer sequence
counts the number of DGCAs with $(0,1)$-coefficient matrices. It
has not yet been included in
OEIS (\url{https://oeis.org/}), and we wonder any other combinatorial interpretations.
\end{remark}

\bigskip
 \noindent
{\bf Conclusion. }
The classification method we adopt for DGCAs is a graded modification of the Skjelbred-Sund method, which highly depends on the diagonally graded assumption.
Furthermore, we expect that such an approach works for other diagonally graded algebraic structures, such as (pre-)Lie algebras~\cite{Ch,KCB}, Novikov algebras~\cite{BMN,BMB,Bd}, and even certain operated algebras like differential algebras, Rota-Baxter algebras~\cite{Guo,GK}, etc.

\bigskip
 \noindent
{\bf Acknowledgements. } We would like to thank Ming Ding and Yue Zhou for helpful discussions, and also Yuxing Shi for the calculation via GAP.
This work is supported by the National Natural Science Foundation of China (Grant No. 12071094, 12171155) and Guangdong Basic and Applied Basic Research Foundation (Grant No. 2022A1515010357).

\bibliographystyle{amsplain}

\end{document}